\documentclass{amsart}

\usepackage{graphicx}
\usepackage{dsfont}
\usepackage{textcomp}
\usepackage{amsmath}
\usepackage{courier}
\usepackage[utf8]{inputenc}
\usepackage{graphicx}
\graphicspath{{images/}}
\usepackage{amsthm}
\usepackage{eucal}
\usepackage{amssymb}
\usepackage{mathrsfs}
\usepackage{dsfont}
\usepackage[usenames,dvipsnames,svgnames,table]{xcolor}
\usepackage{caption}
\usepackage{tikz}
\usepackage{tikz-cd}
\usepackage{wrapfig,lipsum,booktabs}
\usepackage{caption}
\usepackage[square, longnamesfirst,numbers]{natbib}

\usepackage{amstext} % for \text macro
\usepackage{array}   % for \newcolumntype macro
\usepackage[unicode]{hyperref}
\usepackage{verbatim}
\usepackage{makecell}
\usepackage{float}

\newcolumntype{C}{>{$}c<{$}} % math-mode version of "l" column type

\newcommand{\Z}{\mathbb{Z}}

\newcommand{\Q}{\mathbb{Q}}
\renewcommand{\O}{\mathcal{O}}
\renewcommand{\P}{\mathcal{P}}
\newcommand{\NP}{\mathcal{NP}}
\newcommand{\NPC}{\mathcal{NPC}}
\newcommand{\F}{\mathbb{F}}
\newcommand{\p}{\mathfrak{p}}
\newcommand{\q}{\mathfrak{q}}
\newcommand{\m}{\mathfrak{m}}
\renewcommand{\r}{\mathfrak{r}}
\renewcommand{\t}{\mathfrak{t}}
\renewcommand{\u}{\mathfrak{u}}
\renewcommand{\v}{\mathfrak{v}}
\renewcommand{\i}{\text{\textbf{i}}}

\newcommand{\tensor}{\otimes}

\DeclareMathOperator{\mon}{mon}
\DeclareMathOperator{\Hom}{Hom}
\DeclareMathOperator{\Gal}{Gal}

\DeclareMathOperator{\rk}{rk}

\DeclareMathOperator{\Spec}{Spec}
\DeclareMathOperator{\sep}{sep}

\DeclareMathOperator{\HTP}{HTP}
\DeclareMathOperator{\Mat}{Mat}
\let\S\undefined
\DeclareMathOperator{\S}{S}
\DeclareMathOperator{\A}{A}

\newcommand{\ol}[1]{{\overline{#1}}}

\begingroup
\makeatletter
\@for\theoremstyle:=definition,remark,plain\do{%
	\expandafter\g@addto@macro\csname th@\theoremstyle\endcsname{%
		\addtolength{\thm@preskip}{5pt}
		\setlength{\thm@postskip}{10pt}
	}%
}
\endgroup %Zorgt voor mooie enters voor claims

\title{The complexity of root-finding in orders}
\author{Pim Spelier}
\date{\today}

\theoremstyle{plain}
\newtheorem{thm}{Theorem}[section] % reset theorem numbering for each section
\newtheorem{lem}[thm]{Lemma} % same for lemmas
\newtheorem{prop}[thm]{Proposition} % same for vermoedens
\newtheorem{algo}[thm]{Algorithm} % same for vermoedens

\newtheorem{inthm}{Theorem}

\newtheorem{inconj}[inthm]{Conjecture}

\theoremstyle{definition}
\newtheorem{defn}[thm]{Definition} % definition numbers are dependent on theorem numbers
\newtheorem{exmp}[thm]{Example} % same for example numbers

\newtheorem{indefn}[inthm]{Definition}

\theoremstyle{remark}
\newtheorem{rem}[thm]{Remark} % remark numbers are dependent on theorem numbers

\usepackage{chngcntr}
\counterwithin{table}{section}

\begin{document}
%\nocite{*}

\begin{abstract}
    Given an \textit{order}, a commutative ring whose additive group is free of finite rank, a natural computational question is whether a fixed univariate polynomial $f \in \Z[X]$ has a root in this ring. In this paper, we show that the computational difficulty of this depends strongly on the arithmetic properties of $f$. We show that with probability 1, determining whether $f$ has a root is NP-complete. For $\deg f \leq 3$ we give a full classification of the computational complexity: some special $f$ admit a polynomial-time algorithm, and for all other $f$ the problem is NP-complete. Additionally, we prove the problem is undecidable for $f = (X^2+1)^2$, conditional on Hilberts Tenth Problem for $\Q(i)$. The key ingredients for proving NP-completeness are a new source of NP-complete group-theoretic problems developed in \citep{artgroepen}, and a full classification of cubic polynomials with discriminant divisible only by $3$.
\end{abstract}

\maketitle

%-----------------------------------------------------------------------------------------------------------------------

\section{Introduction}
\label{section:intro}
In this paper, we study the computational complexity of finding roots of polynomials in orders. An \textit{order} is a commutative ring $A$ whose underlying additive group is isomorphic to $\Z^n$ for some integer $n \in \Z_{\geq 0}$; this $n$ is called the \textit{rank} of that order, denoted by $\rk A$. Given a polynomial $f \in \Z[x]$, we denote the zero set of $f$ in $A$ by $Z_A(f)$. One can ask about the size of this set, or to specify an element, but the very simplest question is: is this set non-empty? This leads to the following definition.

\begin{indefn}
Let $f \in \Z[X]$ be a polynomial. Then the problem $\Pi_f$ is defined as: given as input\footnote{An order is specified by giving an integer $n$, and structure constructs $(a_{ijk})_{1\leq i,j,k \leq n} \in \Z^{n^3}$ describing the multiplication on $\Z^n$ by $e_i \cdot e_j = \sum_{k = 1}^n a_{ijk} e_k$.} an order $A$, determine whether $Z_A(f)$, is non-empty.
\end{indefn}
% \begin{defn}
% Let $A$ be an order. Then the problem $\Pi_A$ is defined as: given as input a polynomial $f \in \Z[X]$, determine whether $Z_A(f)$ is non-empty.
% \end{defn}
% We use the terminology of polynomial, non-deterministic polynomial, and NP-complete problems to classify these problems; we refer to these classes as $\P,\NP$ and $\NPC$, respectively. A short treatment of the subject can be found in Appendix~A of \citep{artgroepen}.
Recall that a polynomial in $\Z[X]$ is \emph{separable} if it has no double roots in $\overline{\Q}$.

For $f$ separable, the problem $\Pi_f$ is decidable. In fact, the theorem about $\Pi_A$ tells us exactly what happens for reduced $A$.
% \begin{thm}
% \label{thm:pia}
% Let $A$ be a reduced order. Then there is a polynomial time algorithm for $\Pi_A$.
% \end{thm}
\begin{inthm}[Theorem~\ref{thm:pif}]
\label{inthm:pif}
Let $f$ be a separable polynomial. Then $\Pi_f$ lies in $\NP$.
\end{inthm}

\begin{rem}
We warn the reader this is \emph{not} always true for $f$ non-separable. Although for $f$ any polynomial, and $A$ any order, one can prove $Z_A(f) \not= \varnothing$ by specifiying an element $c \in Z_A(f)$ (a certificate for this problem), it is not guaranteed that the size of $c$ is polynomial in the size of the input, and hence it is not guaranteed that $f(c) = 0$ can be proven in polynomial time. Indeed, we show in Theorem~\ref{inthm:undec} that for $f = (X^2+1)^2$ the problem $\Pi_f$ is undecidable (if Hilberts Tenth Problem over $\Q(i)$ is undecidable), while any problem in $\NP$ is decidable.
\end{rem}

The exact behaviour of $\Pi_f$ varies considerably depending on $f$; for example $\Pi_{X^2 + 1}$ is in $\P$, while $\Pi_{X^2+X+1}$ is NP-complete (see Theorem~\ref{thm:quad}). In general, we have the following conjecture.

\begin{inconj}
\label{inconj:pnpc}
Let $f \in \Z[X]$ be separable. Then $\Pi_f$ lies in $\P$ or in $\NPC$.
\end{inconj}

We show that with probability 1, this is true.
\begin{inthm}[Theorem~\ref{thm:sntweemacht}, Theorem~\ref{thm:cub}]
\label{inthm:prob1}
Let $f \in \Z[X]$ be monic of degree $n \geq 2$. Then $\Pi_f$ is NP-complete with probability $1$.
\end{inthm}

Furthermore, for degree $2$ and degree $3$ we prove Conjecture~\ref{inconj:pnpc} explicitly.
\begin{inthm}[Theorem~\ref{thm:quad}]
\label{inthm:quad}
For $f \in \Z[X]$ quadratic monic, we have $\Pi_f \in \P$ if $\Delta(f) = -4$ or $f$ is reducible, and $\Pi_f \in \NPC$ otherwise.  
\end{inthm}
\begin{inthm}[Theorem~\ref{thm:cub}]
\label{inthm:cub}
For $f \in \Z[X]$ cubic monic, we have $\Pi_f \in \P$ if $f$ is reducible, and $\Pi_f \in \NPC$ otherwise.  
\end{inthm}

The case where $f$ is non-separable behaves completely differently. In Section~\ref{section:undec}, we prove the following theorem.
\begin{inthm}[Theorem~\ref{thm:undec}]
\label{inthm:undec}
If Hilberts Tenth Problem over $\Q(\i)$ is undecidable, then the problem $\Pi_{(X^2+1)^2}$ is undecidable.
\end{inthm}

\subsection{Overview of the paper}
The main ingredient in proving NP-completeness is the following type of group-theoretic problem on intersections in Cartesian powers of groups.
\begin{defn}[\protect{\cite[Definition~1.1]{artgroepen}}]
Let $G$ be a finite abelian group, and $S$ a subset of $G$. Define the problem $\Pi_{G,S}$ as follows: on input $(t,H)$ with $t \in \Z_{\geq 0}$ and $H$ a subgroup of $G^t$, is the intersection $S^t \cap H$ non-empty?
\end{defn}
In \cite[Theorem~1.5]{artgroepen} we give a polynomial-time algorithm for certain pairs $(G,S)$, and prove NP-completeness for all other pairs. We summarise the results in Section~\ref{section:group}.

In Section~\ref{section:resu} we prove Theorem~\ref{inthm:pif}, give some examples where $\Pi_f$ is polynomial, and provide some preliminary lemmas relating $\Pi_f$ for different $f$.

We give a quick sketch on the proof of NP-completeness of $\Pi_f$, through reductions from $\Pi_{G,S}$, as worked out in Section~\ref{sec:reductions}. 

The prototypical reduction between problems $\Pi_{G,S} \leq \Pi_f$ is given in Proposition~\ref{prop:npf}. Roughly, one needs to find a specific order $A$ (often we choose a suborder of $\Z[Z_{\ol{\Q}}(f)] \subset \ol{\Q}$) and a finite quotient $\psi: A \to B$. Then one chooses $S = \psi(Z_A(f))$ and $G$ the subgroup of $B$ generated by $S$. Under the further technical condition that $G \cdot G = 0$ (where $\cdot$ denotes the multiplication in $B$) one can assign to each instance $(t,H)$ of $\Pi_{G,S}$ a suborder $A_H$ of $A^t$ such that $Z_{A_H}(f)$ surjects onto $H \cap S^t$. For a surjective map, the target is non-empty if and only if the source is non-empty, and hence we obtain a reduction $\Pi_{G,S} \leq \Pi_f$.

In Proposition~\ref{prop:npf} we have that the quotient $C = B/G$ is non-zero. Note that the image of $Z_A(f)$ in $C$ consists of a single element $0$, which implies that $f$ has a zero of multiplicity $|Z_A(f)|$ modulo $p$ for each prime $p$ dividing $|C|$. In particular, such $p$ divide the discriminant $\Delta(f)$ of $f$. Also note that the set of possible suborders $A$ of $\Z[Z_{\ol{\Q}}(f)] \subset \ol{\Q}$ depends heavily on the Galois group of $f$. Hence the construction of $A,B,G,S$ as above is where arithmetic information, on both the Galois group of $f$ and on its discriminant, enters the picture.

In cases where $f$ has Galois group $S_n$, and $n \ne 3$, and the discriminant is divisible by an odd prime then in Theorem~\ref{thm:gensn} we find $A,B,G,S$ as above to apply Proposition~\ref{prop:npf} to. This is the most general statement of NP-completeness we have, and as a corollorary we obtain Theorem~\ref{inthm:prob1} for $n \ne 3$. For $f$ with small Galois group and/or discriminant divisible only by small primes, we need slightly different reductions from $\Pi_{G,S}$, using more arithmetic information on $f$.

For quadratic polynomials, by Theorem~\ref{thm:gensn} it only remains to check irreducible polynomials of the form $X^2 \pm 2^k$. For $X^2 + 1$, the problem is in fact polynomial by Theorem~\ref{thm:polxnp1}; for all other irreducible polynomials in this family, we give a proof of NP-completeness in Proposition~\ref{prop:x22k}, thus proving Theorem~\ref{inthm:quad}.

For cubic polynomials, in Section~\ref{subs:cubpol} we give a large number of lemmas proving NP-completeness given certain reduction types of $f$ modulo $p$ for $p = 2, p = 3$ and $p>3$. We furthermore use Galois theory, splitting behaviour of primes in extensions of $\Z$, and discriminant bounds to prove that certain other reduction types modulo $2,3$ do not occur. Putting all these lemmas together, we obtain Proposition~\ref{prop:cubpol} showing that for $f$ monic irreducible $\Pi_f$ is NP-complete unless $\Delta(f)$ is of the form $\pm 3^k$.

In Section~\ref{section:disc} we find all cubic monic polynomials with discriminant of the form $\pm 3^k$. Only finitely many (up to equivalence, as defined in Definition~\ref{defn:eqprob}) polynomials then remain, listed with discriminant in Table~\ref{table:diffpol}. These are the polynomials
\[
    X^3-3, X^3-3X+1,X^3-9X+1.
\]
In Section~\ref{section:hard} we finally prove that for these three polynomials, $\Pi_f$ is also NP-complete, finishing the final three cases of \ref{inthm:cub}.

\subsection{Acknowledgements}
This project grew out of the thesis of the author \cite{spelier2018}. It is a pleasure to thank my thesis supervisors Hendrik Lenstra and Walter Kosters for their help. I am also very thankful to Daan van Gent for their comments on the paper.

%%%%%%%%%%%%%%%%%%%%%%%%%%%%%%%%%%%%%%%%%%%%%%%%%%%%%%%%%%%%%%%%%%%%%%%%%%%%%%%%%%%%%%%%%%%%%%%%%%%%%%%%%%%%%%%%%%%%%%%%%%%%%%%%%%%%%%%%%%%%%%%%%%%%%%%%%%%%
\section{Group-theoretic NP-complete problems}
\label{section:group}
We recall the following definitions, remarks and theorems from \citep{artgroepen}.

\begin{defn}
\label{defn:pgs}
Let $R$ be a commutative ring that is finitely generated as a $\Z$-module, let $G$ be a finite $R$-module and $S$ a subset of $G$. Then define the problem $P_{G,S}^R$ as follows. With input $t \in \Z_{\geq 0}, x_* \in G^t$, the $t$-th Cartesian power of $G$, and $H$ a submodule of $G^t$ given by a list of generators, decide whether $(x_* + H) \cap S^t$ is non-empty. Write $P_{G,S}$ for $P_{G,S}^\Z$.
\end{defn}

\begin{defn}
\label{defn:pigs}
Let $R$ be a commutative ring that is finitely generated as a $\Z$-module, let $G$ be a finite $R$-module and $S$ a subset of $G$. Then define the problem $\Pi_{G,S}^R$ as the subproblem of $P_{G,S}^R$ where $x_* = 0$. I.e., with input $t \in \Z_{\geq 0}$ and $H$ a submodule of $G^t$ given by a list of generators, decide whether $H \cap S^t$ is non-empty. Write $\Pi_{G,S}$ for $\Pi_{G,S}^\Z$.
\end{defn}

We will mostly use the case where $R = \Z$ (or $R = \Z/p\Z$), and hence $G,H$ simply are abelian groups. Note that $R,G,S$ are \textit{not} part of the input of the problem. In particular, computations inside $G$ can be done in $O(1)$. 
These problems are certainly in $\NP$, as one can easily give an $R$-linear combination of the generators of $H$, and check that it lies in $S^t - x_*$. We present two theorems from \citep{artgroepen} that completely classify the problems $P_{G,S}$ and $\Pi_{G,S}$, in the sense that for each problem we either have a polynomial time algorithm or a proof of NP-completeness.

\begin{thm}
\label{thm:pgs}
If $S$ is empty or a coset of some subgroup of $G$, then we have $P_{G,S} \in \P$. In all other cases, $P_{G,S}$ is NP-complete.
\end{thm}

\begin{thm}
\label{thm:pigs}
If $S$ is empty or $\theta(S) := \bigcap_{a \in \Z \mid aS \subset S} aS$ is a coset of some subgroup of $G$, then we have $\Pi_{G,S} \in \P$. In all other cases, $\Pi_{G,S}$ is NP-complete.
\end{thm}

\begin{rem}
\label{rem:psi}
Note that if $0 \in S$, then $\theta(S) = \{0\}$; additionally, if $G$ is a group with order a prime power and $S$ does not contain 0, then $\theta(S) = S$, by Lemma~2.15 of \citep{artgroepen}.
\end{rem}

%%%%%%%%%%%%%%%%%%%%%%%%%%%%%%%%%%%%%%%%%%%%%%%%%%%%%%%%%%%%%%%%%%%%%%%%%%%%%%%%%%%%%%%%%%%%%%%%%%%%%%%%%%%%%%%%%%%%%%%%%%%%%%%%%%%%%%%%%%%%%%%%%%%%%%%%%%%%%%%%%%%%%%%%%%%%%%%%%%%%

\section{Positive results on \texorpdfstring{$\Pi_f$}{Pi f}}
\label{section:resu}

In this section we prove some positive results on $\Pi_f$. We prove $\Pi_f \in \NP$ for $f$ separable (Theorem~\ref{thm:pif} below), and we provide some examples with $\Pi_f \in \P$. We furthermore define a notion of equivalence (Definition~\ref{defn:eqprob}) for polynomials $f,g$ that corresponds to equality of the corresponding problems $\Pi_f,\Pi_g$.

First we will give an exponential-time algorithm for $\Pi_f$. %The real work is in the proofs of NP-completeness; we will give a short explanation about the problem in general, including an explanation of when $\Pi_f,\Pi_g$ are equal, some polynomial algorithms and a lemma that allows us to restrict to monic polynomials.

\begin{algo}
\label{algo:np}
We take as input $A$ an order, $f \in \Z[X]$ a non-constant polynomial such that either $f$ is separable or $A$ is reduced. The algorithm returns whether $f$ has a zero in $A$.
\begin{enumerate}
    \item If $f$ is separable, replace $A$ by $A_{\sep}$, the subring consisting of elements of $A$ that are the zero of some separable polynomial in $\Z[X]$, using Algorithm 4.2 of \cite{lenstrasilverberg2017}.
    \item Apply Algorithm 7.2 of \cite{lenstra2018} to $E := A \tensor_{\Z} \Q$ to find irreducible polynomials $g_1,\ldots,g_s \in \Q[X]$ with $E \cong \prod_{i=1}^s K_i$ where $K_i = \Q[X]/(g_i)$, together with an isomorphism $\varphi: \prod_{i=1}^s K_i \to E$.
    \item Use the LLL algorithm \citep{lenstra83} to find $Z_{K_i}(f)$ for every $K_i$.
    \item For every $(\alpha_i)_{i=1}^s \in \prod_{i=1}^s Z_{K_i}(f)$, use the isomorphism $\prod_{i=1}^s K_i \to E$ to compute $\varphi((\alpha_i)_{i=1}^s)$ with respect to the $\Z$-basis $e_1,\ldots,e_{\rk A}$ of $A$, and test whether all coefficients are integral. If all coefficients are integral, then $f$ has a zero in $A$; the answer is yes.
    \item If no zeroes of $f$ in $A$ were found in the previous step, the answer is no.
\end{enumerate}
\end{algo}
\begin{prop}
\label{prop:algo}
The time complexity is 
\[
O\left(p\left(\rk A \deg f \log \left(1 + \hspace{-5.1pt} \sum_{1 \leq i,j,k \leq n} \hspace{-4pt} |a_{ijk}|\right)\right)(\deg f)^{|\Spec(A \tensor_{\Z} \Q)|}\right)
\]
where $p(m) = O(m^\ell)$ for some fixed integer $\ell$.
\end{prop}
\begin{proof}
The adding of $1$ inside the logarithm is done to correctly handle the case $\rk A = 1$. The standard operations as multiplication, addition, all take polynomial time in $(1+\rk A)(1+\deg f) \log \left(1 + \sum_{1 \leq i,j,k \leq n} |a_{ijk}|\right)$. Note that the $s$ we have found in the second step equals $|\Spec(A \tensor_{\Z} \Q)|$, and that in every field of characteristic zero $f$ has at most $\deg f$ zeroes, so we check at most $(\deg f)^{|\Spec(A \tensor_{\Z} \Q)|}$ candidates. Then for each candidate it takes $O((\rk A)^2)$ computations to apply $\varphi$, and time $O(\rk A)$ to compute whether that zero of $f$ in $E$ indeed lies in $A$.
\end{proof}
\begin{rem}
We note that Algorithm~\ref{algo:np} always runs in polynomial time when $|\Spec(A \tensor_\Z \Q)|$ is bounded. In particular, if we fix a separable order $A$, the problem of whether a given $f$ has a zero in $A$ is always solvable in polynomial time.
\end{rem}

% \begin{proof}[Proof of Theorem \ref{thm:pia}]
% Using Algorithm \ref{algo:np} together with Proposition \ref{prop:algo} for non-constant polynomials, and the obvious algorithm for constant polynomials, this theorem is now trivial; in fact, Algorithm \ref{algo:np} works in polynomial time even if we only fix $|\Spec(A \tensor_{\Z} \Q)|$.
% \end{proof}
\begin{thm}
\label{thm:pif}
Let $f$ be a separable polynomial. Then $\Pi_f$ lies in $\NP$.
\end{thm}
\begin{proof}
We need to prove that a yes-instance $A$ of $\Pi_f$ can be certified in polynomial time. Our certification algorithm takes as input $(A,c)$ with $c \in A$, and checks whether $f(c) = 0$. Each yes-instance then has a certificate $c \in Z_A(f)$. Algorithm \ref{algo:np} shows that the size of $c$ is polynomial in the size of $A$, hence our algorithm works in polynomial time. Hence the problem $\Pi_f$ lies in $\NP$.
\end{proof}
\begin{rem}
This does not necessarily work for non-separable polynomials, as then we cannot guarantee that if $Z_A(f)$ is non-empty, it contains an element with length polynomially bounded in the length of the input.
\end{rem}

\begin{lem}
\label{lem:resmon}
Let $f \in \Z[X]$ be non-zero, and let $f_{\mon}$ be its largest degree monic divisor in $\Z[X]$. Then $\Pi_{f} = \Pi_{f_{\mon}}$.
\end{lem}
\begin{proof}
It suffices to show for any order $A$ that $Z_{A}(f) \not= \varnothing$ holds if and only if $Z_A(f_{\mon}) \not= \varnothing$ holds. Obviously, if $f_{\mon}$ has a zero in $A$, then so does $f$. If $f$ has a zero $\alpha$ in $A$, then, as $A$ is an order, $\alpha$ is the zero of some monic polynomial $g$. If we then use again that $A$ is torsion free, we see that $\alpha$ is a zero of the monic polynomial $\gcd(g,f)$. Any monic polynomial that divides $f$ also divides $f_{\mon}$, so $f_{\mon}$ has a zero in $A$. This in fact proves the stronger statement $Z_{A}(f) = Z_A(f_{\mon})$.
\end{proof}

\begin{defn}
\label{defn:eqprob}
Let $f,g \in \Z[X]$ be two polynomials. Then we say that $f$ and $g$ are \emph{equivalent}, notation $f \sim g$, if and only if there exist ring homomorphisms $\varphi: \Z[X]/(f) \to \Z[X]/(g), \psi: \Z[X]/(g) \to \Z[X]/(f)$. 
\end{defn}
\begin{exmp}
For any $f \in \Z[X]$, we have $f \sim f(\pm X + k)$ with $k \in \Z$. 
\end{exmp}
\begin{exmp}
\label{exmp:tweemachteq}
For $n\in \Z_{\geq 1}$, write $n = 2^r s$ with $s$ odd. Then $X^{n}+1 \sim X^{2^r}+1$.
\end{exmp}

Per the functioral bijection between $Z_A(f)$ and $\Hom(\Z[X]/(f),A)$, the following lemma follows.

\begin{prop}
Let $f,g \in \Z[X]$ be two monic polynomials. Then $\Pi_f = \Pi_g$ holds if and only if $f$ and $g$ are equivalent.
\end{prop}

Now we will treat the few polynomial cases known so far. We start with a rather trivial lemma. Recall that a problem is called trivial if all instances are yes-instances or all instances are no-instances; note that trivial problems always lie in $\P$.
\begin{lem}
\label{lem:trivpol}
Let $f \in \Z[X]$ be a polynomial with $Z_\Z(f) \not= \varnothing$. Then $\Pi_f$ is trivial.
\end{lem}
\begin{proof}
By the condition $Z_\Z(f) \not= \varnothing$, there is a homomorphism $\Z[X]/(f) \to \Z$, hence $\Hom(\Z[X]/(f),A)$ is non-empty for any order $A$.
\end{proof}

There is one family of polynomials for which a non-trivial polynomial time algorithm is known, as proven in the following theorem.
\begin{thm}
\label{thm:polxnp1}
Let $n \in \Z_{\geq 1}$. Then for $f = X^n + 1$ we have $\Pi_f \in \P$.
\end{thm}
\begin{proof}
A zero of $f$ is necessarily a root of unity. By Theorem 1.2 of \citep{lenstrasilverberg2017}, we can find a set of generators $S$ for $\mu(A)$, the group of roots of unity. Then asking whether $f$ has a root in $A$ is asking whether in $\mu(A)$ the element $-1$ is an $n$-th power, i.e., if $-1$ is in the subgroup generated by $\{s^n \mid s\in S\}$. Theorem 1.3 of the mentioned article allows us to compute this in polynomial time, hence this gives a polynomial time algorithm for $\Pi_f$.
\end{proof}
\begin{rem}
Note that we have found a polynomial time algorithm for $\Pi_{\Phi_n}$ with $\Phi_n$ the $n$-th cyclotomic polynomial, where $n$ is a power of two. Strangely enough, Theorem \ref{thm:genxn} will tell us that for $X^2 + X + 1$, the third cyclotomic polynomial, the problem is NP-complete as $(X+1)^2 + (X+1) + 1 \equiv X^2 \bmod 3$.
\end{rem}

\section{Reductions and NP-completeness}
\label{sec:reductions}
Now we will prove two general theorems that can be used to classify problems $\Pi_f$ as NP-complete. First we will state a general proposition that we will use multiple times to prove NP-completeness; although it cannot be applied in every proof in Section \ref{subs:quadpol} and Section \ref{subs:cubpol}, it is the prototype for all proofs.
\begin{prop}
\label{prop:npf}
Let $f \in \Z[X]$ be a polynomial, $A$ an order, $\psi: A \to B$ a surjective ring homomorphism with $B$ finite, $R$ a subring of $B$, and $G$ an $R$-module inside $B$. Assume that $G \cap R = 0$ and the multiplication on $B$ restricted to $G\times G$ is the zero map. Let $a \in R$ such that $\psi(Z_{A}(f)) = a + S$ with $S \subset G$. Then $\Pi_{G,S}^R \leq \Pi_f$.
\end{prop}
\begin{proof}
Let $(t,H)$ be an instance of $\Pi_{G,S}^R$. Note that $R$ has a unique $R$-linear ring homomorphism into $B^t$, the diagonal map. We write this as an inclusion; in that way, we have $R[H] \subset B^t$. By the condition that multiplication on $G$ is the zero map and $G \cap R = 0$ we have that $R[H]$, the subring of $B^t$ generated by $R$ and $H$, is (under the condition $t>0$) as an $R$-module isomorphic to $R \oplus H$. Now we see that $R[H] \cap (a + S^t)$ is in bijection with $H \cap S^t$, by the map $x \mapsto x-a$. Let $A_H \subset A^t$ be the inverse image of $R[H]$ with respect to the map $A^t \to B^t$; as $R[H]$ is a ring, so is $A_H$. We see that we end up with a surjective map $Z_{A_H}(f) \to H \cap S^t$. A surjective map has the property that the domain is empty if and only if the codomain is empty, hence $H \cap S^t \not= \varnothing$ if and only if $Z_{A_H}(f) \not= \varnothing$. So we produce $A_H$ as an instance of $\Pi_f$, completing the reduction.
\end{proof}
\begin{rem}
Note that if $R = \Z \cdot 1 \subset B$, then an $R$-module is just an abelian group $G$ that satisfies $|R|G = 0$, with no further structure. Hence then $\Pi_{G,S}^R $ equals $\Pi_{G,S}$.
\end{rem}

\begin{thm}
\label{thm:genxn}
Let $f$ be a monic irreducible polynomial over $\Z$ of degree $n>1$, and $p \nmid n$ a prime such that $f\equiv X^n \bmod p$. Then $\Pi_f$ is NP-complete.
\end{thm}
\begin{proof}
We will use Proposition~\ref{prop:npf}.

Let $\alpha_1,\ldots,\alpha_n$ be the zeroes of $f$ in $\overline{\Q}$, and let $A$ be the order $\Z[\alpha_1,\ldots,\alpha_n]$. Let $I$ be the $A$-ideal generated by $\alpha_1,\ldots,\alpha_n$.

Now let $B := A/(pA + I^2)$, which is non-zero as $f\equiv X^n \bmod p$, let $R = \F_p \subset B$, let $\overline{\alpha_i}$ be the image of $\alpha_i$ in $B$, let $G = \langle \overline{\alpha_1}, \ldots,\overline{\alpha_n} \rangle$, and $S = \{\overline{\alpha_1},\ldots,\overline{\alpha_n}\}$. We will prove that $\Pi_{G,S}^{\F_p} = \Pi_{G,S}$ is NP-complete.

As $G$ is a group with order a power of $p$, by Theorem \ref{thm:pigs} and Remark \ref{rem:psi} it suffices to check that $0\not\in S$ and that $S$ is not a coset.

By the condition on $f \bmod p$, we see that $J := I + pA$ is nilpotent in $A/pA$. As $A = \Z[I]$ we have $A/pA = \F_p[J]$. Since $n > 1$, we have $\rk A \geq 2$ hence $|A/pA| \geq p^{2}$, which implies that $J \not= \{0\}$. As $J$ is nilpotent, that implies that $J^2 \subsetneq J$. So at least one of $\overline{\alpha_1},\ldots,\overline{\alpha_n}$ is non-zero, and by the transitivity of the Galois action, all of them are non-zero.

We have to prove that $S$ is not a coset. Since a coset has $p$-power cardinality, it suffices to prove that $|S|>1$ and $|S| \mid n$. Assume $|S| = 1$. Then in $B$, we have $\overline{\alpha_1} = \cdots = \overline{\alpha_n}$; as $\alpha_1 + \cdots + \alpha_n \in p\Z$ we have $n\overline{\alpha_1} = 0$. We know $n$ is a unit in $\F_p$, hence $\overline{\alpha_1} = 0$, contradiction. The fact $|S| \mid n$ follows immediately from the action of the Galois group. As said, we find that $S$ is not a coset.

Now note that $G \cdot G = 0$ and $G \cap \F_p = 0$, so with $a = 0$ all of the conditions of Proposition~\ref{prop:npf} are satisfied. Together with the NP-completeness of $\Pi_{G,S}$, this implies that $\Pi_f$ is NP-complete.
\end{proof}

\begin{prop}
\label{prop:gencomp}
Let $f$ be a monic irreducible polynomial over $\Z$ of degree $n>1$ and $p\mid \Delta(f)$ an odd prime. Let $A$ be an order with $\alpha_1,\alpha_2$ two distinct zeroes of $f$ in $A$. Further, let $a\in \overline{\F_p}$, let $\F_q = \F_p(a)$ and let $\psi: \Z[\alpha_i,\alpha_j] \to \F_q[X]/(X^2) = \F_q[\varepsilon]$ be a ring homomorphism such that $\psi(\alpha_1) = a + \varepsilon$ and $\psi(\alpha_2) = a - \varepsilon$. Finally assume, that we have that $Z_{A}(f) = \{\alpha_1,\alpha_2\}$ or we have both that $Z_{\Z[\alpha_1]}(f) = \{\alpha_1\}$ and that all zeroes in $Z_A(f) \setminus \{\alpha_1,\alpha_2\}$ get sent under $(\F_q[\varepsilon] \to \F_q) \circ \psi$ to something different from $a$. Then $\Pi_f$ is NP-complete.
\end{prop}
\begin{proof}

We first consider the case that $Z_{A}(f) = \{\alpha_1,\alpha_2\}$; we will directly use Proposition~\ref{prop:npf}. Let $B =  \F_q[\varepsilon]$, let $R = \F_q$. Now let $G = \varepsilon \F_q, S = \{\pm \varepsilon\}$. As $G \cdot G = 0$ and $G \cap R = 0$, all of the conditions of Proposition~\ref{prop:npf} hold, hence $\Pi_{G,S}^R \leq \Pi_f$. By Lemma~2.10 of \citep{artgroepen} we see that $\Pi_{G,S}^R$ is NP-complete as $p > 2$, hence so is $\Pi_f$.

We now consider the second case; we have to slightly change the proof as we cannot use Proposition~\ref{prop:npf} directly. We still reduce from the problem $\Pi_{\F_q,\{\pm 1\}}^{\F_q}$, which is NP-complete by Lemma~2.10 of \citep{artgroepen}. Let $(t,H)$ be an instance of $\Pi_{\F_q,\{\pm 1\}}^{\F_q}$. Let $B =  \F_q[\varepsilon]$, let $R = \F_q$, and let $R_H = R[\varepsilon\F_q \times \varepsilon H] \subset B^{t+1}$. Let $\varphi: \Z[\alpha_1] \times A^t \to B^{t+1}$ be $\psi$ on every coordinate, and let $A_H = \varphi^{-1}(R_H)$. We want to find which elements in $R_H$ are the image under $\varphi$ of a zero of $f$. Such an element is of the form $x + \varepsilon\cdot(y,h)$ with $x  \in R \subset B^t,y \in \F_q$ and $h \in H$. Since $Z_{\Z[\alpha_1]}(f) = \{\alpha_1\}$, on the first coordinate we must get $a + \varepsilon$, meaning $x = a$ and $y = 1$. On the last $t$ coordinates, we then have $a + \varepsilon h$. By assumption any zero $\alpha \in Z_A(f) \setminus \{\alpha_1,\alpha_2\}$ is sent under $\psi$ to $b + c\varepsilon$ with $b \not= a$. Hence the fact that $a + \varepsilon(y,h)$ is the image under $\varphi$ of some zero of $f$ in $A_H$, is equivalent to it lying in the image under $\varphi$ of some element in $\{\alpha_1\} \times \{\alpha_1,\alpha_2\}^{t}$. So we see that $Z_{A_H}(f)$ is non-empty if and only if $H \cap S^t$ is non-empty.
\end{proof}

For the proof of the second general theorem we first state a definition, following Section 1.6 of \cite{gioga13}.
\begin{defn}
\label{defn:gioia}
Let $R$ be a commutative ring, and fix $f \in R[X]$ monic of degree $n$. Then we define $A_0 = R, f_0 = f$ and recursively for $0 \leq i < n$ we define $A_{i+1} = A_i[X_{i+1}]/(f_i(X_{i+1})), \alpha_{i+1} = \overline{X_{i+1}} \in A_{i+1}$ and $f_{i+1}(X) = \frac{f_i(X)}{X-\alpha_{i+1}}$ as element of $A_{i+1}[X]$.
\end{defn}
We will only use this definition in the case $R = \Z$. Note that if the Galois group of $f$ over $\Q$ is $\S_n$, then $A_i$ is isomorphic to $\Z[\beta_1,\ldots,\beta_i]$ where $\beta_1,\ldots,\beta_n$ are the zeroes of $f$ in $\overline{\Q}$. For any other Galois group, this is never the case for $i = n$, as the rank of $A_n$ is $n!$ while the rank of $\Z[\beta_1,\ldots,\beta_n]$ is $|\Gal(f)| < n!$. Furthermore, note that in $A_i[X]$ we have $\prod_{j=1}^i (X- \alpha_i)\mid f$, and by Theorem 1.6.7 of \cite{gioga13}, the ring $A_i$ is universal with this property, i.e., in the case $R = \Z$ we have that $A_i$ represents the functor from commutative rings to sets $S \mapsto \{(s_1,\ldots,s_i) \in S^i \mid \prod_{j=1}^i (X-s_j) \text{ divides } f \text{ in } S[X]\}$.

\begin{thm}
\label{thm:gensn}
Let $f$ be a monic irreducible polynomial over $\Z$ of degree $n \geq 2$ with either the Galois group acting triply transitively on $Z_{\overline{\Q}}(f)$ or $n=2$. Let $p$ be an odd prime factor of $\Delta(f)$. If $n = 3$, assume that $f$ has a zero of multiplicity $2$ modulo $p$. Then $\Pi_f$ is NP-complete.
\end{thm}
\begin{proof}
This proof works by showing that the conditions of Proposition \ref{prop:gencomp} hold, where we choose $a \in \overline{\F_p}$ to be a zero of $f$ of multiplicity at least 2 (or exactly 2 if $n = 3$). Write $\F_q$ for $\F_p(a)$.

First, we construct the map $\psi$ as needed. Our $A$ will be $A_2$. Let $\alpha_1,\alpha_2$ be the roots $x_1,x_2$ of $f$ in $A_2$. As $\Gal(f)$ acts triply transitively on $Z_{\overline{\Q}}(f)$ or $n=2$, we have that $A_2$ is isomorphic to $\Z[\alpha,\beta]$ where $\alpha,\beta$ are any two zeroes of $f$ in $\overline{\Q}$. We use the universal property of $A_2$ to construct $\psi: A_2 \to \F_q[\varepsilon]$ with $\psi(\alpha_1) = a + \varepsilon, \psi(\alpha_1) = a - \varepsilon$; the universal property implies that this map exists, since $(X-(a+\varepsilon))(X-(a-\varepsilon)) = (X-a)^2$, which divides $f$ modulo $p$ exactly because $a$ is a double zero of $f$.

Now, let $n \neq 3$. Then the conditions tell us that $A_2$ only contains $\alpha_1,\alpha_2$ and no other roots of $f$, which means the conditions for Proposition \ref{prop:gencomp} hold. If $n = 3$, then as $f_1$ is irreducible over $\Z[\alpha_1]$ we have that $Z_{\Z[\alpha_1]}(f) = \{\alpha_1\}$, and as $a$ is a root of multiplicity 2, we can also apply Proposition \ref{prop:gencomp}.
\end{proof}
As a consequence, we have the following theorem on the average behaviour of $\Pi_f$.
\begin{thm}
\label{thm:sntweemacht}
Fix $n \geq 2$ an integer differing from $3$, and let $f$ be a random monic degree $n$ polynomial in $\Z[X]$. Then with probability $1$, we have that $\Pi_f$ is NP-complete.
\end{thm}
\begin{proof}
Note that with probability $1$, the polynomial $f$ will be irreducible with Galois group $S_n$. By Minkowski's theorem, the discriminant will never be a unit. 
It immediately follows that for $n \neq 3$ the only polynomials with Galois group $\S_n$ for which we have not proven NP-completeness yet, are those with discriminant $\pm 2^k$. Since the discriminant is not of the form $\pm 2^k$ with probability 1, we see that if the degree $n \geq 2$ is fixed, then the problem $\Pi_f$ is almost surely NP-complete.
\end{proof}
\begin{rem}
For $n=3$, this result will also hold by Theorem~\ref{thm:cub}.
\end{rem}

\subsection{Quadratic polynomials}
\label{subs:quadpol}
In this section we will prove Theorem~\ref{thm:quad} by fully treating the quadratic case. Let $f \in \Z[X]$ be a quadratic monic polynomial. If $f$ is reducible, then by Lemma \ref{lem:trivpol} the problem $\Pi_f$ is in $\P$. If $f$ is irreducible and there is an odd prime dividing $\Delta(f)$, then we can use Theorem \ref{thm:genxn} or Theorem \ref{thm:gensn} to find that $\Pi_f$ is NP-complete. The only case that remains is $f$ irreducible, $\Delta(f) = \pm 2^k$ with $k \in \Z_{\geq 0}$. Since an irreducible polynomial of degree $\geq 2$ has a prime factor in its discriminant by Minkowski's theorem, we have $k \geq 1$. Hence $f$ has a double root modulo 2, so the coefficient of $X$ is even. By translating, we may assume $f = X^2 - a$, with discriminant $4a$.

Hence the only polynomials that remain are of the form $X^2 - a$ with $|a|$ a power of 2 and $a$ not a square. For $X^2+1$, we have given a polynomial time algorithm in Theorem \ref{thm:polxnp1}. In all other cases, we have $2 \mid a$ and we can use the following theorem.
\begin{prop}
\label{prop:x22k}
Let $f = X^2 - a$ with $2 \mid a$ and $a$ not a square. Then $\Pi_f$ is NP-complete.
\end{prop}
\begin{proof}
We will reduce from $\Pi_{\Z/8\Z,\{\pm 1\}}$, which is NP-complete by Theorem \ref{thm:pigs}. Let $A = \Z[\sqrt{a}]$ and $B = \Z/8\Z[\sqrt{a}], R = \Z/8\Z \subset B$; let $S \subset B$ be $\{\pm\sqrt{a}\}$ and $G = (\Z/8\Z) \cdot \sqrt{a} \subset B$. Now let $(t,H)$ be an instance of $\Pi_{\Z/8\Z,\{\pm 1\}}$, let $C$ be the $B$-algebra $\{(x_1,\ldots,x_t) \in B^t \mid x_1 \equiv \dots \equiv x_t \bmod 2B\}$. As $\sqrt{a} \equiv -\sqrt{a} \bmod 2B$, we have $S^t \subset C$. Letting $H' = H\sqrt{a} \cap C$, we see that $(t,H)$ is a yes-instance if and only if $H' \cap S^t$ is non-empty. Note that $H' = \langle(\sqrt{a} + 2B^t) \cap H' \rangle \cup (2B^t \cap H')$; if $H' \subset 2B^t$, this is trivially a no-instance. Otherwise, we see  $H' = \langle(1 + 2B^t) \cap H' \rangle$ hence $H' \cdot H' \cdot H'$ is generated by elements of the form $\prod_{i=1}^3 (\sqrt{a}+\sqrt{a}x_i)$ with $x_1,x_2,x_3 \in 2B^t$ and $\sqrt{a}+\sqrt{a}x_1,\sqrt{a}+\sqrt{a}x_2,\sqrt{a}+\sqrt{a}x_3 \in H'$. Because $4a = 0$ in $B$, this product equals $a\sqrt{a}(1 + x_1 + x_2 + x_3)$. Now using that $ax_3 = -ax_3$, we see this equals $a(\sqrt{a} + \sqrt{a}x_1 + \sqrt{a} + \sqrt{a}x_2 - (\sqrt{a} + \sqrt{a}x_3)) \in H'$. This implies that $H' \cdot H' \cdot H'$ is a subset of $H'$, which means that $\Z/8\Z[H']$ is as an additive group $\Z/8\Z + H' + H'\cdot H'$, with $\Z/8\Z[H'] \cap S^t = H' \cap S^t$. Then we define $A_H$ to be the inverse image of $\Z/8\Z[H']$ under the natural map $A^t \to B^t$, and we see that $A_H$ contains a zero of $f$ exactly if $H' \cap S^t \neq \varnothing$, completing the reduction.
\end{proof}

All in all, we have shown the following theorem.
\begin{thm}
\label{thm:quad}
For $f \in \Z[X]$ quadratic monic, we have $\Pi_f \in \P$ if $\Delta(f) = -4$ or $f$ is reducible, and $\Pi_f \in \NPC$ otherwise.  
\end{thm}

\subsection{Cubic polynomials}
\label{subs:cubpol}
In this part we will prove NP-completeness for many monic cubic polynomials. Note that a reducible monic cubic polynomial $f$ has a zero in $\Z$, and hence $\Pi_f$ is trivial according to Lemma \ref{lem:trivpol}. Therefore we will consider only irreducible monic polynomials. To concisely state our many lemmas, we first state a definition, using the terminology of Definition \ref{defn:gioia}.
\begin{defn}
\label{defn:zrank}
Let $f \in \Z[X]$ be monic irreducible cubic. We define the \emph{$\Z$-rank} of $f$, written $\rk_\Z(f)$, to be the rank of the smallest $A_i$ which contains three zeroes of $f$.
\end{defn}

Note that we have $\rk_\Z(f) = 6$ if $f_1$ is irreducible over $A_1$, and $\rk_\Z(f) = 3$ otherwise.

If the $\Z$-rank of some cubic monic irreducible polynomial $f$ is $6$ and $\Gal(f) = \A_3$, then one can check that $A_2 \tensor_\Z \Q$ contains 9 zeroes of $f$; the following important lemma controls the number of zeroes in $A_2$.
\begin{lem}
\label{lem:rk6lz}
Let $f$ be monic irreducible cubic, with $\Z$-rank $6$. Then $f$ has exactly three zeroes in $A_2$.
\end{lem}
\begin{proof}
If $\Gal(f)$ is $\S_3$, then the statement is trivial, as then $A_2$ is isomorphic to $\Z[Z_{\overline{\Q}}(f)]$. From now on, assume $\Gal(f) = \A_3$, with $\Gal(\Q(\alpha)/\Q) = \langle \sigma \rangle$. Note $f$ has at least three zeroes $\alpha_1,\alpha_2,\alpha_3$ in $A_2$, obtained by the construction of $A_2$. Let $\alpha, \beta = \sigma(\gamma), \gamma = \sigma(\beta)$ be the zeroes of $f$ in $\overline{\Q}$, where we pick our algebraic closure of $\Q$ such that $\alpha = \alpha_1$. Note that $A_2 \tensor_\Z \Q$ is naturally isomorphic as an $A_1$-algebra to $\Q(\alpha) \times \Q(\alpha)$, with $\alpha_2 \tensor 1$ being sent to $(\beta,\gamma)$. Under this isomorphism, $A_1 = \Z[\alpha]$ is sent to $\Z[\alpha] \subset \Q(\alpha) \times \Q(\alpha)$ by the diagonal. All in all we have the injections in the following diagram
\[ 
\begin{tikzcd}
A_2 = \Z[\alpha,\alpha_2] \arrow [r] & \Q(\alpha) \times \Q(\alpha) \\
A_1 = \Z[\alpha] \arrow [r] \arrow [u] & \Q(\alpha) \arrow [u] \\
\end{tikzcd}
\vspace{-18pt}
\] 
Now we define an equivalence relation on the nine zeroes of $f$ in $\Q(\alpha) \times \Q(\alpha)$ by $x \sim y$ if the fields they generate inside $\Q(\alpha) \times \Q(\alpha)$ are the same. The nine zeroes fall into three equivalence classes: the corresponding fields are $K_i = \{(x,y) \in \Q(\alpha) \times \Q(\alpha) \mid y = \sigma^i(x)\}$ for $i = 0,1,2$, each isomorphic to $\Q(\alpha)$. Specifically, we see that $\alpha_2 \not\sim \alpha$, and hence $\alpha_3 \not\sim \alpha,\alpha_2$ by the symmetry. We claim that of each equivalence class, only one zero lies in $\Z[\alpha,\alpha_2]$. By the symmetry, we only need to prove it for the equivalence class containing $\alpha$, consisting of $(\alpha,\alpha),(\beta,\beta),(\gamma,\gamma)$. Note that $\Z[\alpha,\alpha_2]$ has basis $1,\alpha_2$ as $\Z[\alpha]$-module and $\Q(\alpha) \times \Q(\alpha)$ has basis $1,\alpha_2$ as $\Q(\alpha)$-module. So $\Z[\alpha,\alpha_2] \cap  \Q(\alpha) = \Z[\alpha]$, and hence $\beta,\gamma$ are not in $\Z[\alpha,\alpha_2]$ as they are by hypothesis not in $\Z[\alpha]$. This proves that each equivalence class contains only one zero in $\Z[\alpha,\alpha_2]$, so $\Z[\alpha,\alpha_2] = A_2$ has exactly three zeroes of $f$.
\end{proof}

\begin{lem}
\label{lem:cubquad}
Let $f \in \Z[X]$ be monic irreducible cubic of $\Z$-rank $6$. If there is a prime $p$ with $p \not= 2$ such that $f$ has a zero of multiplicity $2$ modulo $p$, then $\Pi_f$ is NP-complete.
\end{lem}
\begin{proof}
If $\Gal(f) = \S_3$, then this is a special case of Theorem \ref{thm:gensn}. For $\Gal(f) = \A_3$, we can prove the NP-completeness directly from Proposition~\ref{prop:gencomp} and Lemma~\ref{lem:rk6lz}; we will check the conditions of Proposition~\ref{prop:gencomp}. Write $f \equiv (X-a)^2(X-b) \bmod p$ with $a \not\equiv b \bmod p$. As in the proof of Theorem \ref{thm:gensn} we take $A = A_2$ and use that $f$ has exactly three zeroes $\alpha_1,\alpha_2,\alpha_3$ in $A$. Then we construct by the universal property of $A$ a ring homomorphism $\psi: A \to \F_3[\varepsilon]$ with $\psi(\alpha_1) = a + \varepsilon$ and $\psi(\alpha_2) = a-\varepsilon$. By $\alpha_1 + \alpha_2 + \alpha_3 = 2a + b$ we have $\psi(\alpha_3) = b$. Also, note that $Z_{\Z[\alpha_1]}(f) = \{\alpha_1\}$ since $f$ has $\Z$-rank 6.
\end{proof}

\begin{lem}
\label{lem:cubcub}
Let $f \in \Z[X]$ be monic irreducible cubic. If there is a prime $p$ with $p \not=3$ and $f \equiv (X-a)^3 \bmod p$ for some $a \in \F_p$, then $\Pi_f$ is NP-complete.
\end{lem}
\begin{proof}
Special case of Theorem \ref{thm:genxn}.
\end{proof}

\begin{lem}
\label{lem:rk3ord2modp}
Let $f \in \Z[X]$ be monic irreducible cubic of $\Z$-rank $3$. Then there is no prime $p$ such that $f$ has a zero of multiplicity $2$ modulo $p$.
\end{lem}
\begin{proof}
Let $\alpha$ be one of the zeroes of $f$ in $\overline{\Q}$. Note that since $\Z[\alpha]$ already contains the other two zeroes of $f$, the Galois group of $f$ acts on $\Z[\alpha]$. Let $p$ be a rational prime. We will show that $\Gal(f)$ acts transitively on the primes above $p$; if it did not, there would be $\p\mid p,\q\mid p$ with $\p,\q$ in different $\Gal(f)$-orbits. Then using the Chinese remainder theorem, there is an element $x$ such that $x \in \p$, but not in any $\sigma(\q)$ for $\sigma \in \Gal(f)$. Then $N_{K/\Q}(x) = \prod_{\sigma \in \Gal(f)} \sigma(x)$ is contained in $\p$ but not in $\q$; but $N_{K/\Q}(x) \in \Z$ and both $\p,\q$ have intersection $(p)$ with $\Z$, so that is impossible.

Hence, all primes over $p$ must be isomorphic. In particular $f \bmod p$ factors as a product of polynomials, all of the same degree. We conclude the proof by observing that $2$ does not divide $3$.
\end{proof}
\begin{rem}
This lemma becomes false if one replaces the rank condition by the condition $\Gal(f) = \A_3$. For example take $X^3 + 6X^2 - X - 5$ with discriminant $65^2$; modulo $5$ this factors as $X(X+3)^2$.
\end{rem}

\begin{prop}
\label{prop:first}
Let $f \in \Z[X]$ be monic irreducible cubic. If $\Delta(f) \not= \pm 2^k 3^{\ell}$ with $k,\ell \in \Z_{\geq 0}$, then $\Pi_f$ is NP-complete. Further, if $f \equiv (X-a)^3 \bmod 2$ or $f \equiv (X-a)(X-b)^2 \bmod 3 $ with $b \not\equiv a \bmod 3$, then $\Pi_f$ is NP-complete as well.
\end{prop}
\begin{proof}
If $\Delta(f)$ contains a prime factor $p > 3$, then $f$ has a zero of multiplicity $3$ or $2$ modulo $p$. In the first case, $\Pi_f$ is NP-complete by Lemma \ref{lem:cubcub}. In the second case, by contraposition of Lemma \ref{lem:rk3ord2modp} we have $\rk_\Z(f) = 6$ and hence by Lemma \ref{lem:cubquad} the problem $\Pi_f$ is NP-complete.

Furthermore, if $f \equiv (X-a)^3 \bmod 2$ or $f \equiv (X-a)(X-b)^2 \bmod 3 $ with $b \not= a$ then we can again use respectively Lemma \ref{lem:cubcub} or Lemma \ref{lem:rk3ord2modp} followed by Lemma \ref{lem:cubquad}.
\end{proof}

\begin{lem}
\label{lem:disc23}
Let $f \in \Z[X]$ be monic irreducible cubic such that $f$ has a zero of multiplicity $2$ modulo $2$ and one of multiplicity $3$ modulo $3$. Then $\Pi_f$ is NP-complete.
\end{lem}
\begin{proof}
Note that by Lemma \ref{lem:rk3ord2modp} the $\Z$-rank of $f$ is $6$. Let $A$ be the $A_2$ corresponding to $f$, and let $\alpha_1,\alpha_2,\alpha_3$ be the three zeroes of $f$ in $A_2$ given by Lemma \ref{lem:rk6lz}. As in the proof of Theorem \ref{thm:gensn} using that $f$ has a zero $a$ of order $2$ modulo $2$, we find that for any $t \in \Z_{\geq 0}$ there is a homomorphism $\varphi: A_1 \times A_2^t \to \F_2$ obtained from applying on each coordinate the morphism $\psi:A_2 \to  \F_2$ where one sends both $\alpha_1$ and $\alpha_2$ to $a$. Then one can check that under $\psi$ the zero $\alpha_3$ is sent to $a + 1$. We let $A' \subset A_1 \times A_2^t$ be the inverse image of $\F_2 \subset \F_2^{t+1}$; the zeroes of $f$ in $A'$ are exactly $\{\alpha_1\} \times \{\alpha_1, \alpha_2\}^t$. Let $a'$ be a zero of $f$ of multiplicity $3$ modulo $3$ and let $\psi'$ be the ring homomorphism $A_2 \to \F_3[\varepsilon]$ given by $\psi'(\alpha_1) = a + \varepsilon, \psi'(\alpha_2) = a- \varepsilon$. We reduce from $\Pi_{\F_3,\{\pm 1\}}$; letting $(t,H)$ be an instance of $\Pi_{\F_3,\{\pm 1\}}$ and $R_H = \F_3[\varepsilon\F_3 \times \varepsilon H]$ we see that the inverse image of $R_H$ with respect to $A' \to \F_3[\varepsilon]^t$ contains a zero of $f$ if and only if $H \cap \{\pm 1\}^t$ is non-empty, completing the reduction. As $\Pi_{\F_3,\{\pm 1\}}$ is NP-complete, this completes the proof.
\end{proof}

\begin{lem}
\label{lem:cubdis2macht}
There are no irreducible cubic polynomials with discriminant $\pm2^k$ with $k \in \Z_{\geq 0}$.
\end{lem}
\begin{proof}
Let $f$ be such a polynomial --- we will derive a contradiction. Let $\Z[\alpha] = \Z[X]/(f)$ with $\alpha = \overline{X}$, let $K = \Q(\alpha)$ and let $\Delta$ be the discriminant of $K$ (and note that $\Delta$ is also a power of 2, up to sign). We now have the following inclusion of fields:
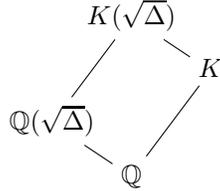
\[ 
\begin{tikzcd}[cramped,column sep={{{{3em,between origins}}}}, row sep={{{{2em,between origins}}}}]
 & K(\sqrt{\Delta}) \arrow [ldd,-] \arrow [rd, -]&  \\
 & & K \arrow [ldd,-]\\
\Q(\sqrt{\Delta}) \arrow [rd, -]& & \\
 & \Q & \\
\end{tikzcd}
\]
Using that the Minkowski bound is at least 1, we find $\Delta$ is in absolute value at least 13. However, the discriminant of $\Q(\sqrt{\Delta})$ is one of $1,-4,\pm 8$ (it is $1$ exactly if $\Gal(f) = \A_3$). From this we will derive a contradiction, using the discriminant of $K(\sqrt{\Delta})$ in between. We do this by looking at the splitting behavior of $(2)$.

Since $2\mid\Delta$, the prime $(2)$ ramifies over $K/\Q$. We see that in $\O_K$ either $(2) = \p^3$ or $(2) = \p^2\q$ with $\p\not=\q$. In the first case, since then $\p$ ramifies tamely, we have $2^2\|\Delta$, so $\Delta = \pm 4$, contradiction with the upper bound $|\Delta| \geq 13$ we found earlier. The other case is a bit more complex. Note that in this case $f$ has Galois group $\S_3$ as $K/\Q$ is clearly not Galois; in a Galois extension, all ramification indices of a prime over $2$ are equal. That $K/\Q$ is not Galois implies that the discriminant of $\Q(\sqrt{\Delta})$ is not 1, so it is divisible by 2. Hence $(2)$ factors as $\r^2$ in  $\Q(\sqrt{\Delta})$. Since $K(\sqrt{\Delta})/\Q$ is a Galois extension, we see that in $K(\sqrt{\Delta})$ we have $(2) = (\t\u\v)^2$ with $\t\u\v = \r$ and $\t\u = \p$ and $\v^2 = \q$. We see that $K(\sqrt{\Delta})/\Q(\sqrt{\Delta})$ is unramified, and hence $\Delta_{K(\sqrt{\Delta})} = \Delta_{\Q(\sqrt{\Delta})}^3$. Note we also have $\Delta_{K(\sqrt{\Delta})} \geq \Delta^2$. Now we make another small case distinction: if $\Delta_{\Q(\sqrt{\Delta})} = -4$, we find $|\Delta| \leq 8$, contradiction. If $\Delta_{\Q(\sqrt{\Delta})} = \pm 8$, we find $|\Delta| \leq 22$, but $\Delta$ is a power of two with an odd number of factors 2 and it is in absolute value at least 13, and we again arrive at contradiction.

We conclude that there is no cubic number field with discriminant $\pm2^k$, so also no irreducible cubic polynomial with such a discriminant.
\end{proof}

We again summarise the results in a proposition.
\begin{prop}
\label{prop:second}
Let $f \in \Z[X]$ be monic irreducible cubic. If $\Delta(f)$ has a prime factor other than $3$ or $f$ does not have a triple zero modulo $3$, then $\Pi_f$ is NP-complete.
\end{prop}
\begin{proof}
If $\Delta(f)$ has a prime factor bigger than $3$, the problem is already NP-complete by Proposition \ref{prop:first}; from now on, assume that it does not have such a prime factor. If $\Delta(f)$ is divisible by $2$, then by contraposition of Lemma \ref{lem:cubdis2macht} it is also divisible by $3$, and unless $f$ has a zero of multiplicity $2$ modulo $2$ and a zero of multiplicity $3$ modulo $3$, the problem is NP-complete by Proposition \ref{prop:first}; if we are in that case, we can use Lemma \ref{lem:disc23} to prove NP-completeness. This proves the first part of the statement.

If $|\Delta(f)|$ is a power of $3$, then it is divisible by $3$ by the Minkowski bound $|\Delta(f)| \geq 13$. If it does not have a triple zero modulo $3$, it must have a zero of multiplicity $2$, which means that the problem is NP-complete by Proposition \ref{prop:first}.
\end{proof}
We finish this section with two lemmas that tell us what happens if the polynomial has a triple zero modulo a power of $3$, for the $\Z$-rank 6 and 3 cases separately.

\begin{lem}
\label{lem:zr6tr9}
Let $f \in \Z[X]$ be monic irreducible cubic with $\Z$-rank $6$ with a triple root modulo $9$. Then $\Pi_f$ is NP-complete.
\end{lem}
\begin{proof}
Assume by translation that $f \equiv X^3 \bmod 9$. Let $B = (\Z/9\Z)[\omega,\varepsilon]$, where $1 + \omega + \omega^2 = 0$ and $\varepsilon^2 = 0$. Let $R = \Z/9\Z$ and let $a = 0$. Letting $\alpha_1,\alpha_2,\alpha_3$ be the three zeroes of $f$ in $A_2$ (guaranteed by Lemma~\ref{lem:rk6lz}), we use the universal property of $A_2$ to give a map $A_2 \to R$ sending $\alpha_1$ to $\varepsilon$ and $\alpha_2$ to $\omega\varepsilon$ and $\alpha_3$ to $\omega^2\varepsilon$. Taking $R = \Z/9\Z \subset R$ and $G = \Z/9\Z[\omega]\varepsilon$, we can now reduce from the problem $\Pi_{\Z/9\Z[\omega]\varepsilon,\{\varepsilon,\omega\varepsilon,\omega^2\varepsilon\}}$ by Proposition~\ref{prop:npf}; this problem is NP-comple by Theorem~\ref{thm:pigs}.
\end{proof}

\begin{lem}
\label{lem:zr3tr27}
Let $f \in \Z[X]$ be monic irreducible cubic with $\Z$-rank $3$. Then $f$ does not have a triple root modulo $27$.
\end{lem}
\begin{proof}
We will argue by contradiction. Let $f$ be as in the conditions, and assume by translating that $f$ has $0$ as a triple root modulo 27. Let $\alpha,\beta,\gamma$ be the zeroes of $f$ in $\overline{\Q}$. We define $R := \Z[\alpha]/(27) \cong (\Z/27\Z)[\eta]$ where $\eta^3 = 0$. As $f$ splits as $(X-\alpha)(X-\beta)(X-\gamma)$ in $\Z[\alpha]$, we find that $X^3$ totally splits over $R$ with one of the factors being $X-\eta$. It can be seen that if $X^3$ factors over $R$ as $(X-\eta)(X-a)(X-b)$ then $(X-a)(X-b) = X^2 + \eta X + \eta^2$. Since $X^2 + \eta X + \eta^2$ splits over $R$, the discriminant $-3\eta^2$ is a square of $R$. Let $x \in R$ be such that $-3\eta^2 = x^2$ (in fact, we can take $x = a-b$). Let $\m = (3,\eta)$ be the maximal ideal in $R$. We see $-3\eta^2 \in \m^3 \setminus \m^4$, hence $x \in \m \setminus \m^2$. But if $3u + v \eta$ is an element of $\m \setminus \m^2$, then $u$ or $v$ is a unit, and in both cases the square is in $\m^2 \setminus \m^3$ as $9,3\eta,\eta^2$ form a basis of the $\F_3$ vector space $\m^2/\m^3$. This means that $-3\eta^2$ is not a square, contradiction, so no such $f$ exists.
\end{proof}

These lemmas all together give us the following proposition.

\begin{prop}
\label{prop:cubpol}
Let $f$ be monic irreducible cubic. Then $\Pi_f$ is NP-complete if at least one of the following conditions holds:
\begin{itemize}
    \item $\Delta(f)$ has a prime factor other than $3$;
    \item $f$ does not have a triple zero modulo $3$;
    \item $\rk_\Z(f) = 6$ and $f$ has a triple zero modulo $9$;
    \item $f$ has a triple zero modulo $27$.
\end{itemize}
\end{prop}
\begin{proof}
This proposition consists of four statements; the first two are given by Proposition \ref{prop:second}, the third by Lemma \ref{lem:zr6tr9}, and the fourth by the contraposition of Lemma \ref{lem:zr3tr27} followed by Lemma \ref{lem:zr6tr9}. 
\end{proof}

%%%%%%%%%%%%%%%%%%%%%%%%%%%%%%%%%%%%%%%%%%%%%%%%%%%%%%%%%%%%%%%%%%%%%%%%%%%%%%%%%%%%%%%%%%%%%%%%%%%%%%%%%%%%%%%%%%%%%%%%%%%%%%%%%%%%%%%

\section{Cubic polynomials with discriminant \texorpdfstring{$\pm 3^{\ell}$}{only divisible by 3}}
\label{section:disc}

In the previous section we have proven that for any cubic monic irreducible polynomial $f \in \Z[X]$ whose discriminant has a prime factor that is not $3$, the problem $\Pi_f$ is NP-complete. This motivates the following theorem; the exact conditions of the theorem complement Proposition \ref{prop:cubpol}, in the sense that if $f$ cubic monic irreducible does not satisfy these conditions, then it holds that $\Pi_f \in \NPC$ by Proposition \ref{prop:cubpol}. We refer to Definitions \ref{defn:zrank} and \ref{defn:eqprob} for the definitions of $\Z$-rank and equivalence of polynomials respectively.

\begin{thm}
\label{thm:diffpol}
Let $f \in \Z[X]$ be monic irreducible cubic, with discriminant of the form $\pm 3^k$ with $k \in \Z_{\geq 0}$. Assume that $f$ has a zero of multiplicity $3$ modulo $3$, and not a triple zero modulo $27$. Also, assume that if the $\Z$-rank is $6$, then $f$ does not have a triple zero modulo $9$. Then $f$ is equivalent to one of the polynomials in Table \ref{table:diffpol}.
\end{thm}

For the proof of Theorem \ref{thm:diffpol}, we first state a definition and a trivial lemma about integral points on a family of elliptic curves.
Throughout the rest of the section, we take $S = \{3\}$, and denote the $S$-integers $\Z[S^{-1}]$ as $\Z_S$.
\begin{defn}
Let $a \in \Z \setminus \{0\}$. Then $C_a$ is the elliptic curve given by the equation $y^2 = x^3 + a$.
\end{defn}

\begin{lem}
\label{lem:parp}
Let $a,k \in \Z$. Then there is a bijection $C_{a}(\Q) \to C_{ak^6}(\Q)$ given by $(x,y) \mapsto (k^2 x, k^3 y)$; if $k$ is a power of $3$, this induces a bijection $C_{a}(\Z_S) \to C_{ak^6}(\Z_S)$
\end{lem}
\begin{proof}
Both statements follow immediately from the calculation \[(k^3 y)^2 - (k^2 x)^3 - k^6 a = k^6(y^2 - x^3 - a).\]
\end{proof}

\begin{proof}[Proof of Theorem \ref{thm:diffpol}]
%Finding diophantine equations for the troublesome cubic polynomials.
Let $f$ be a cubic irreducible polynomial with discriminant $\pm 3^{\ell}$ with $\ell \geq 1$, satisfying the conditions of the theorem. Since $f$ has a triple zero modulo $3$, the coefficient corresponding to $X^2$ is divisible by $3$. Hence we can put $f$ into the form $X^3 + pX + q$ with $p,q \in \Z$ by translation.

Now $\Delta(f)$ has the simple formula $-4p^3-27q^2$. Setting $-4p^3-27q^2 = \pm 3^{\ell}$, we find a family of elliptic-curve-like diophantine equations. Multiplying such an equation by $2^4 3^3$, and substituting $x = -2^2 3p, y = 2^2 3^3 q$, we find the equation 
\[
C_{\mp 2^{4} 3^{\ell'}}: y^2 = x^3 \mp 2^{4} 3^{\ell'}
\]
with $\ell' = \ell + 3$. To find all possible polynomials up to equivalence, it suffices to find all integral points $(x,y)$ on one of these curves. Lemma~\ref{lem:parp} will us do even more: we can parametrise all points on $\bigcup_{\ell \geq 0, s  = \pm1} C_{s 2^4 3^{\ell}}(\Z_S)$ by $\bigcup_{6 > \ell \geq 0, s  = \pm1} C_{s 2^4 3^{\ell}}(\Z_S) \times \Z_{\geq 0}$. Now theorem 4.3 of \citep{silverman2009} tells us there are only finitely many $S$-integral points on the curves $C_{\pm 2^4 3^{\ell}}$ with $0 \leq \ell < 6$. The author used Sage \citep{sagemath} to explicitly find these points. The list of parametrised corresponding polynomials up to the transformation $f(X) \mapsto -f(-X)$ can be seen in Table \ref{table:pol}. The reducible polynomials are those with Galois group of cardinality $1$ or $2$. Next to the irreducible polynomials are the values of $t \in \Z_{\geq 0}$ such that the polynomial has integral coefficients. Now we observe that all polynomials have a triple root modulo 27 for $t \geq 2$, and that $X^3 + 9$ and $X^3-54X+153$ both have a triple zero modulo 9 and $\Z$-rank 6. This almost give the final Table \ref{table:diffpol}; it only remains to observe that $X^3 -3X + 1 \sim X^3 - 21X + 37$, as for $f \in \{X^3 -3X + 1,X^3 -21X + 37\}$ we have that $\Z[X]/(f)$ is the ring of integers of $\Q(\zeta_9 + \zeta_9^{-1})$ where $\zeta_9$ is a primitive ninth root of unity; $\zeta_9 + \zeta_9^{-1}$ is a zero of $X^3 -3X + 1$, and $3(\zeta_9 + \zeta_9^{-1})^2 + \zeta_9 + \zeta_9^{-1} - 6$ is a zero of $X^3 -21X + 37$.

Furthermore, note that the three polynomials in Table \ref{table:diffpol} are pairwise non-equivalent, as all of the discriminants are different.
\end{proof}

\renewcommand{\arraystretch}{1.3}
\begin{table}[H]
\centering
\begin{tabular}{ |c |c |c| }
\hline
Polynomial & \makecell{Cardinality of \\ Galois group} & \makecell{All $t \in \Z_{\geq 0}$ for which\\ the polynomial is integral} \\ \hline
$ X^{3} - \frac{1}{3} \cdot 3^{2 \, t} X + \frac{1}{27} \cdot 3^{3 \, t} $ & $3$ & $\geq 1$ \\
$ X^{3} - \frac{73}{108} \cdot 3^{2 \, t} X + \frac{595}{2916} \cdot 3^{3 \, t} $ & $1$ & -\\
$ X^{3} - \frac{7}{3} \cdot 3^{2 \, t} X + \frac{37}{27} \cdot 3^{3 \, t} $ & $3$ & $\geq 1$ \\
$ X^{3} - 3^{2 \, t} X + \frac{1}{3} \cdot 3^{3 \, t} $ & $3$ & $\geq 1$ \\
$ X^{3} - \frac{193}{12} \cdot 3^{2 \, t} X + \frac{2681}{108} \cdot 3^{3 \, t} $ & $2$ & -\\
$ X^{3} + \frac{1}{27} \cdot 3^{3 \, t} $ & $2$ & $\geq 1$ \\
$ X^{3} - \frac{1}{12} \cdot 3^{2 \, t} X + \frac{7}{108} \cdot 3^{3 \, t} $ & $6$ & -\\
$ X^{3} + \frac{1}{9} \cdot 3^{3 \, t} $ & $6$ & $\geq 1$ \\
$ X^{3} + \frac{2}{3} \cdot 3^{2 \, t} X + \frac{7}{27} \cdot 3^{3 \, t} $ & $2$ & $\geq 1$ \\
$ X^{3} + \frac{1}{3} \cdot 3^{3 \, t} $ & $6$ & $\geq 1$ \\
$ X^{3} - \frac{3}{4} \cdot 3^{2 \, t} X + \frac{5}{12} \cdot 3^{3 \, t} $ & $6$ & -\\
$ X^{3} - 6 \cdot 3^{2 \, t} X + \frac{17}{3} \cdot 3^{3 \, t} $ & $6$ & $\geq 1$ \\
\hline
\end{tabular}
\caption{A list containing all monic cubic polynomials in $\Z[X]$, up to the substition $f(X) \mapsto -f(-X)$,  that have a triple zero modulo 3 and discriminant of the form $\pm 3^k$, together with the Galois group.}
\label{table:pol}
\end{table}

\begin{table}[H]
\centering
\begin{tabular}{ |c |c |c| }
\hline
Polynomial & Discriminant  & Factorisation of discriminant\\ \hline %\hline
$X^3-3$ & $-243$&  $-3^5$ \\
$X^3 - 3X + 1$ & $81$&  $3^4$ \\
$X^3 - 9X + 9$ & $729$ & $3^6$ \\ \hline
\end{tabular}
\caption{A minimal set $S$ of polynomials such that every monic cubic irreducible polynomial that satisfies the conditions of Theorem \ref{thm:diffpol} is equivalent to a polynomial in $S$.}
\label{table:diffpol}
\end{table}

%%%%%%%%%%%%%%%%%%%%%%%%%%%%%%%%%%%%%%%%%%%%%%%%%%%%%%%%%%%%%%%%%%5

\section{NP-completeness for the remaining cubic polynomials}
\label{section:hard}
In this section we prove NP-completeness for the problems $\Pi_f$ with $f$ in Table \ref{table:diffpol}, at the end concluding the proof of Theorem \ref{thm:cub}.

\begin{lem}
Let $f = X^3 - 3$. Then $\Pi_f$ is NP-complete.
\end{lem}
\begin{proof}
Let $\alpha,\beta,\gamma$ be the three zeroes of $f$ in $\overline{\Q}$. Let $B$ be the finite ring $\Z/9\Z[\pi] := \Z/9\Z[X]/(X^6 + 3)$. Note $B$ has a $\Z/3\Z$-grading $B = B_0 \oplus B_1 \oplus B_2$ with $B_i = \pi^i \Z/9\Z \oplus \pi^{i+3} \Z/9\Z$. We denote $\zeta := -\frac12 + \frac12 \pi^3$; observe that $\zeta^2 + \zeta + 1 = 0$. In this ring, $X^3 - 3$ has a factorisation as $(X+\pi^2)(X + \zeta \pi^2)(x + \zeta^2 \pi^2)$. As $\Gal(f) = \S_3$, the order $A := \Z[\alpha,\beta,\gamma]$ is naturally isomorphic to to $A_2$. Then by the universal property of $A_2$, we have a morphism $\psi: A \to B$ given by $\psi(\alpha) = -\pi^2, \psi(\beta) = -\zeta\pi^2,\psi(\gamma) = -\zeta^2\pi^2$. Letting $S = \{\psi(\alpha),\psi(\beta),\psi(\gamma)\}$, we note $S \subset B_2$. We define $G = B_2$. Note that $S$ is not a coset in $G$ and does not contain zero. By Theorem \ref{thm:pigs} and Lemma \ref{rem:psi}, this means $\Pi_{G,S}$ is NP-complete.

We will give a reduction $\Pi_{G,S} \leq \Pi_f$. Let $(t,H)$ be an instance of $\Pi_{G,S}$. Let $C$ be the subring of $B^t$ given by $C = \{(x_1,\ldots,x_t) \in B^t \mid x_1 \equiv \dots \equiv x_n \bmod \zeta-1\}$. Note $S^t \subset C$, as $\psi(\alpha) \equiv \psi(\beta) \equiv \psi(\gamma) \bmod \zeta - 1$. Let $H'$ be $H$ intersected with $C$. We may assume $H'$ is not contained in $(\zeta-1)B^t$; if it is, clearly $H' \cap S^t = \varnothing$. Note that in $B_2/(\zeta -1)B_2$ we have $\pi^2(\zeta - 1) = 3\pi^2 + 5\pi^5 = 0$ and $\pi^5(\zeta - 1) = -3\pi^2 + 3\pi^5 = 0$. We deduce that $B_2/(\zeta -1)B_2 = \{0,\pi^2, 2\pi^2\}$. Writing $H' = \langle H' \cap (\pi^2 + (\zeta-1) \pi^2 B_0^t) \rangle \cup (H' \cap (\zeta-1)\pi^2 B_0^t)$, we see $H' = \langle H' \cap (\pi^2 + (\zeta-1)\pi^2 B_0^t) \rangle$. That means that $H'\cdot H' \cdot H' \cdot H'$ is generated by elements of the form $\prod_{i=1}^4 (\pi^2 + (\zeta-1) \pi^2x_i)$ with $x_i\in B_0,\pi^2 + (\zeta-1) \pi^2 x_i \in H'$ for $i = 1,\ldots,4$. Using that $\pi^6 = -3$ and that the exponent of $B$ equals 9, we see that the product equals $-3\pi^2\left(1 + \sum_{i=1}^4 x_i (\zeta-1)\right)$, which equals $-3\sum_{i=1}^4 (\pi^2 + (\zeta-1)\pi^2 x_i) \in H'$. Hence $H'\cdot H' \cdot H' \cdot H' \subset H'$, meaning that $\Z/9\Z[H']$ is equal to $\Z/9\Z + H' + H'\cdot H' + H'\cdot H'\cdot H'$. Using the grading of $B$, the intersection $\Z/9\Z[H'] \cap S^t$ hence equals $H'$ itself. Now we can define $A_H$ to be the inverse image under $A^t \to B^t$ of $\Z/9\Z[H']$, and we see $Z_f(A_H)$ is non-empty exactly if $H \cap S^t$ is non-empty. This completes the reduction.
\end{proof}

\begin{lem}
Let $f = (X-1)^3 - 3(X-1) + 1 = X^3 - 3 X^2 + 3$. Then $\Pi_f$ is NP-complete.
\end{lem}
\begin{proof}
Let $R$ be the ring $\F_3[\varepsilon] = \F_3[X]/(X^2)$, and let $G$ be a free $R$-module of rank $1$, with generator $m$. Let $S = \{0,m,-m-\varepsilon m\} \subset G$. Note that $P_{G,S}^{R}$ is NP-complete, as by Lemma~2.6 of \citep{artgroepen} with $\varphi: x \mapsto m - x$ we have $P_{G,\{0,m\}}^{R} \leq P_{G,S}^{R}$, and by Lemma~2.10 we know $P_{G,\{0,m\}}^{R}$ is NP-complete. Now we define a new problem $P$: the input is $t \in \Z_{> 0}$, a sub-$R$-module $H$ of $G^t$ and $x_* \in G^t$ with $(m,\ldots,m) - \varepsilon x_* \in H$; the output is whether $(x_* + H) \cap S^t$ is non-empty. Obviously, $P \leq P_{G,S}^R$.
We will also prove $P_{G,S}^R \leq P$. Let $(t,H,x_*)$ be an instance of $P_{G,S}^R$. For ease of notation, from now on we write $\underline{x}$ for a vector consisting of all $x$'es. Let $t' = t+1, H' = H\times\{0\} + R\cdot (\underline{m}-\varepsilon(x_*,0)) $ and $x_*' = (x_*,0)$. Then $(t',H', x_*')$ is an instance of $P$. If $(t,H,x_*)$ is a yes-instance of $P_{G,S}^R$ with $h \in H$ such that $h + x_* \in S^t$, then $(h,0) \in H'$ and $(h,0) + (x_*,0) \in S^{t'}$, so $(t',H', x_*')$ is a yes-instance of $P$.
Conversely, if $(t',H',x_*')$ is a yes-instance of $P$, then there is an $h' \in H'$ with $h' + (x_*,0) \in S^{t'}$. By looking at the last coordinate, we see that $h'$ can be written as $(h,0) + v(\underline{m}-\varepsilon(x_*,0))$ with $h \in H, v\in R$. Note that $\sigma: G \to G,x \mapsto (1+\varepsilon)(x-m)$ gives a bijection on $S$ which cycles $S$. If we also denote $\sigma$ for the map $G^{t'} \to G^{t'}$ that applies $\sigma$ coordinatewise, we see that $\sigma(x_*' + H') = x_*' + (1+\varepsilon)H' - (1+\varepsilon)(\underline{m}-\varepsilon x_*')$ which implies that $\sigma(x_*' + H')$ lies in $x_*' + H'$. Then clearly $\sigma$ also acts on $(x_*' + H') \cap S^{t'}$. Therefore without loss of generality $(h,0) + v(\underline{m}-\varepsilon(x_*,0)) + (x_*,0)$ is zero on the last coordinate, meaning $v = 0$. Then $(h,0) + (x_*,0) \in S^{t'}$ hence $h + x_* \in S^t$, so we find $(t,H,x_*)$ is a yes-instance of $P_{G,S}^R$. We have now proven that $P \approx P_{G,S}^R$, so we have $P \in \NPC$.

We will now reduce from $P$ to $\Pi_f$. Let $\alpha$ be a zero of $f$ in $\overline{\Q}$, and note that with $A := \Z[\alpha]$ we have $Z_A(f) = \{\alpha,\alpha^2-2\alpha,\alpha-\alpha^2+3\}$. Let $\p$ be the prime ideal in $A$ over 3 generated by $\alpha$; then $(3)$ factorises as $\p^3$. Define $B = A/\p^4$. Note that as $-3 = \alpha^2(\alpha - 3)$ we have $-3 = \alpha^3$ in $B$. Finally, note that $Z_A(f)$ can be written as $\alpha + \{0,\alpha^2-3\alpha,-\alpha^2 + 3\}$ where $\{0,\alpha^2-3\alpha,-\alpha^2 + 3\}$ is a subset of $\p^2$, and modulo $\p^4$ it is equal to $\{0,\alpha^2,-\alpha^2-\alpha^3\}$. This means the image of $Z_A(f)$ in $B$ is $\alpha + \{0,\alpha^2,-\alpha^2-\alpha^3\}$.

We take $m$, the generator of $G$, to be $\alpha^2 \in B$, with $\varepsilon \in R$ acting on $G$ as multiplication by $\alpha$. Let $(t,H,x_*)$ be an instance of $P$. Now define $R_H \subset B^t$ as $\Z/9\Z + \Z/3\Z(\underline{\alpha} + x_*) + H$. Note that this is in fact a ring; the only non-trivial requirement is that $(\underline{\alpha} + x_*)^2 \in R_H$, but $(\underline{\alpha}+x_*)^2 = \underline{\alpha^2} + 2\alpha x_* = \underline{m} - \varepsilon x_*$, which is an element of $H$ by definition of the problem $P$. Also, note that $\underline{3} \in R_H$ is $\underline{-\alpha^3}$, and $\underline{\alpha^3} = \varepsilon(\underline{m} - \varepsilon x_*) \in H$. This tells us that $R_H \cap G^t = H$. Now we see that $(\underline{\alpha} + S^t) \cap R_H = (\underline{\alpha} + x_*) + (-x_* + S^t) \cap R_H$ is in bijection with to $(-x_* + S^t) \cap R_H = (-x_* + S^t) \cap H$. Let $A_H$ be the inverse image under $A^t \to B^t$ of $R_H$. As $\alpha + S$ is the image of $Z_A(f)$ in $B$, we see $Z_f(A_H)$ is non-empty exactly if $(x_* + H) \cap S^t$ is non-empty. This completes the reduction.
\end{proof}

\begin{lem}
Let $f = X^3-9X+9$. Then $\Pi_f$ is NP-complete.
\end{lem}
\begin{proof}
Let $\alpha$ be a zero of $f$ in $\overline{\Q}$, and note that with $A = \Z[\alpha]$ we have $Z_A(f) = \{\alpha,\alpha+\alpha^2-6,-2\alpha-\alpha^2+6\}$. Let $B = A/(9,3\alpha^2), R = \Z/9\Z \subset B$. We see $B$ is a ring of cardinality $3^5$, generated as an additive group by $1,\alpha,\alpha^2$ of order $9,9,3$ respectively. To prove NP-completeness, let $G = \langle\alpha^2+3\rangle \subset B$. This is a group of cardinality $3$. Let $S = \{\pm (\alpha^2+3)\} \subset G$. Note that $S$ is not a coset and does not contain 0, so $\Pi_{G,S}$ is NP-complete. We will reduce from this problem to $\Pi_f$. Let $(t,H)$ be an instance of $\Pi_{G,S}$. Write $T$ for the image of $Z_A(f)$ in $B$.

Let $t' = 2t+1$. Let $H' = \{(x,-x,0) \mid x \in H\} \subset B^{t'}$. Let $x_* = (\alpha - (\alpha^2+3),\ldots,\alpha - (\alpha^2+3),\alpha)$, and let $R_H = R[H',x_*]$. Note that as an additive group, this is generated by $\Z/9\Z,H',x_*,x_*^2,H'x_*$. We see $x_*$ and $x_*^2$ have order $9$ and $3$ respectively.

Claim: $R_H \cap T^{t'} $ is non-empty if and only if $H \cap S^t $ is non-empty. We prove this by examining an element $x \in R_H \cap T^{t'}$. Let $\sigma: B \to B, x \mapsto x^2 + x + 3$ and $\tau: B \to B, x \mapsto -2x-(x^2+3)$ be two maps, and note that by virtue of the Galois group still acting on $T$, we have that $\sigma,\tau$ induce transitive permutations on $T$, and $\sigma|_T = \tau|_T^{-1}$. Denoting by $\sigma$ and $\tau$ the two maps $B^{t'} \to B^{t'}$ that coordinatewise perform $\sigma$ respectively $\tau$ it is clear that $\sigma(R_H),\tau(R_H)$ are subsets of $R_H$, as $R_H$ is a ring. This tells us that $\sigma(x),\tau(x)$ also lie in $R_H \cap T^{t'}$. Letting $\pi: B^{t'} \to B$ denote the projection onto the last coordinate, we see that this means that $R_H \cap T^{t'}$ is non-empty if and only if $R_H \cap T^{t'} \cap \pi^{-1}(\alpha)$ is non-empty.

As $\pi(H') = 0$, we have that $\pi(R_H) = \pi(\Z/9\Z + \Z/9\Z x_* + \Z/3\Z x_*^2)$. As $\pi(1) = 1, \pi(x_*) = \alpha, \pi(x_*^2) = \alpha^2$ we see that $R_H \cap \pi^{-1}(\alpha) = x_* + H' + H'x_*$. Finally, we will prove that $(x_* + H' + H'x_*) \cap T^{t'}$ is non-empty if and only if $H \cap S^t$ is non-empty. Note that $x_* + H' + H'x_*$ contains an element of $T^{t'}$ if and only if there are $h_1,h_2 \in H'$ with $x_* + h_1 + h_2x_* \in T^{t'}$. Writing $h_1 = (x,-x,0)$ and $h_2 = (y,-y,0)$ with $x,y \in H$, this is equivalent to $(x,-x) + (y,-y)(\alpha - (\alpha^2+3)) \in \{\pm(\alpha^2 +3),-3\alpha\}^{2t}$. As $(\alpha^3+3)G = 0$, we can write $(x,-x) + (y,-y)(\alpha - (\alpha^2+3) = (x,-x) + \alpha(y,-y)$ with $\alpha G = \langle 3\alpha \rangle$. Then we see that $(y,-y)$ is an element of $\{0,-(\alpha^2+3)\}^{2t}$, implying $y = 0$. So we find $R_H \cap T^{t'} \not= \varnothing \Leftrightarrow \exists x \in H : (x,-x) \in \{\pm(\alpha^2 + 3)\}^{2t}$. This is clearly equivalent to $H \cap S^t \not= 0$, proving the claim.

That means that we have constructed a subring $R_H \subset B^{t'}$ such that $R_H \cap T^{t} \not= \varnothing \Leftrightarrow H \cap S^{t} \not=\varnothing$ holds. Letting $A_H$ be the inverse image of $R_H$ under the natural map $A^{t'} \to B^{t'}$, we have completed the reduction.
\end{proof}

\begin{thm}
\label{thm:cub}
For $f \in \Z[X]$ cubic monic, we have $\Pi_f \in \P$ if $f$ is reducible, and $\Pi_f \in \NPC$ otherwise.  
\end{thm}
\begin{proof}
Let $f \in \Z[X]$ be cubic, monic. By Lemma \ref{lem:trivpol} we have $\Pi_f \in \P$ if $f$ is reducible. Assume that $f$ is irreducible. If it satisfies the conditions of Proposition \ref{prop:cubpol}, then $\Pi_f$ is NP-complete by that proposition. Otherwise, it satisfies the conditions of Theorem \ref{thm:diffpol}, and hence is equivalent to one of the three polynomials in Table \ref{table:diffpol}. For those three polynomials NP-completeness has been proven in this section. This completes the proof.
\end{proof}

\section{An undecidability result}
\label{section:undec}
In this section we prove the undecidability of $\Pi_{(X^2+1)^2}$, contingent on the undecidability of Hilberts Tenth Problem over $\Q(\i)$. We first give a short definition
\begin{defn}
We define $\HTP(R,S)$ with $R$ a ring and $S \subset R$ to be: given an $n \in \Z_{\geq 1}$ and a set of polynomials $P$ in $R[X_1,\ldots,X_n]$, determine whether the polynomials in $P$ have a common zero in $S^n$. We write $\HTP(R)$ for $\HTP(R,R)$.
\end{defn}
We start by rewriting $\HTP(\Q(\i))$ with the following definition and theorem.
\begin{defn}
We define $v: \Q(\i) \to \frac12\Z \cup \{\infty\}$ as the extension of the $2$-adic valuation on $\Q$ with $v(2) = 1$. Write $A$ for the ring of $(1+i)$-adic Gaussian integers.
\end{defn}
\begin{thm}
\label{thm:eqhtp}
The problem $\HTP(\Q(\i))$ is equivalent with $\HTP(\Q(\i),A^*)$.
\end{thm}
\begin{proof}
We prove the reduction $\HTP(\Q(\i),A^*) \leq \HTP(\Q(\i))$ by separably proving $\HTP(\Q(\i),A^*) \leq \HTP(\Q(\i),A))$ and $\HTP(\Q(\i),A) \leq \HTP(\Q(\i))$. For the first one, we can model a unit of $A$ by adding for every variable $X_j$ occuring the polynomial $X_jY_j=1$. For the second reduction, we use Lemmas 6, 9 and 10 of \citep{robinson}. Let $\p = (1+\i)$, and let $\q_1,\q_2$ be two different primes in the inverse ideal class of $\p$ with $\p,\q_1,\q_2$ all distinct (this is possible by Lemma 6 of the article), and let $a_1,a_2$ be such that $(a_j) = \p\q_j$ for $j =1,2$. Let $b_1,b_2$ be as given by the proof of Lemma 9. Then by Lemma 10, the equation $1-a_jb_jc_j^3=x^2-a_jy^2 -b_jz^2$ has a solution in $x,y,z$ if and only if $c_j$ is a $\p$-adic and a $\q_j$-adic integer. Since $\q_1,\q_2$ are disctinct, the equations $c = c_1 + c_2,1-a_1b_1c_1^3=x_1^2-a_1y_1^2 -b_1z_1^2,1-a_2b_2c_2^3=x_2^2-a_2y_2^2 -b_2z_2^2$ model that $c$ is a $(1+\i)$-adic integer. This completes the inequality $\HTP(\Q(\i),A^*) \leq \HTP(\Q(\i))$.

For the inequality $\HTP(\Q(\i)) \leq \HTP(\Q(\i),A^*)$, note that the expression $\frac{x+y}{z^2+7}$ takes on every value in $\Q(\i)$ for $x,y,z \in A^*$; this is clear from $A^* + A^* = A$ and the fact that $-7$ is $2$-adically a square, implying that $z^2+7$ can have arbitrarily small valuations. That means that we can replace each variable $X_i$ by $\frac{x_i+y_i}{z_i^2+7}$, clearing out denominators, to find an equivalent system of equations for the problem $\HTP(\Q(\i),A^*)$.
\end{proof}

We first slightly alter the definition of $\HTP$ to a slightly less usual but more useful form.
\begin{defn}
We define $\HTP'(R,S)$ with $R$ a ring and $S \subset R$ to be: given an $n \in \Z_{\geq 1}$ and a set of polynomials $P$ in $R[X_1,\ldots,X_n]$ of degree at most $2$, determine whether the polynomials in $P$ have a common zero in $S^n$. We write $\HTP'(R)$ for $\HTP'(R,R)$.
\end{defn}
\begin{thm}
\label{thm:mydprm}
Take any ring $R$ and $S \subset R$. The problems $\HTP(R,S)$ and $\HTP'(R,S)$ are equivalent.
\end{thm}
\begin{proof}
The reduction $\HTP'(R,S) \leq \HTP(R,S)$ is trivial from the definition.

We will now show $\HTP(R,S) \leq \HTP'(R,S)$. Let $(n,P)$ be an input for $\HTP(R,S)$. We briefly sketch the reduction, producing $(m,Q)$ such that the polynomials in $Q$ have a common zero exactly if $P$ has a zero.
\begin{enumerate}
    \item Let $m := n, Q := P$.
    \item While $Q$ contains a polynomial $q$ containing a monomial $c \prod_{i=1}^k X_{n_i}, c \not= 0$ where $n_i \in \{1,\ldots,m\}$ for $i = 1,\ldots,k$ of degree $k$ strictly bigger than $2$, make $m := m + 1, Q := Q \cup \{X_m - X_{n_1} X_{n_2}\}$ and in $q$ replace the monomial $c \prod_{i=1}^k X_{n_i}$ with $c X_m \prod_{i=3}^k X_{n_i}$, lowering the degree of that monomial.  
\end{enumerate}
Note that the zero set of $Q$ is conserved in each step, and that step 2 always terminates. This proves the theorem.
\end{proof}

\begin{thm}
\label{thm:undec}
If Hilberts Tenth Problem over $\Q(\i)$ is undecidable, then the problem $\Pi_{(X^2+1)^2}$ is undecidable.
\end{thm}
\begin{proof}
By Theorems \ref{thm:eqhtp} and \ref{thm:mydprm}, we reduce from the problem $\HTP'(\Q(\i),A^*)$. Let $n \in \Z_{\geq 1}$ and $P = \{p_1,\ldots,p_m\}$ a subset of $\Q(\i)[X_1,\ldots,X_n]$ consisting of polynomials of degree at most 2 be given. By removing denominators, assume $P \subset \Z[\i,X_1,\ldots,X_n]$. We will construct input order $B$ for $\Pi_{(X^2+1)^2}$ that is a yes-instance if and only if the polynomials in $P$ have a common zero in $(A^*)^n$.

Embed $\Z[\i,X_1,\ldots,X_n]$ in $\Z[\i,X_0,\ldots, X_n]$. We now multiply every monomial in one of the polynomials of $S$ by a power of $X_0$ such that the polynomial is homogeneous of degree $2$; call the resulting homogeneous polynomials $q_1,\ldots,q_m$. For $1\leq k \leq m$ let $C_{k}$ be twice the matrix corresponding to the quadratic form $q_k$. Note $C_k \in \Mat(n+1,\Z[\i])$. Letting $X = (X_0,\ldots,X_n)$, we then have $X^\top C_{k} X = 2 q_k(X_0,\ldots,X_n)$ and $q_k(1,X_1,\ldots,X_n) = p_k (X_1,\ldots,X_n)$.

Let $1,v_0,v_1,\ldots,v_n,w_1,\ldots,w_m$ be formal variables and define $\Z[\i]$-modules $V = \bigoplus_{k=0}^n v_k\Z[\i]$ and $W = \bigoplus_{k=1}^m w_k\Z[\i]$. We then choose $B'$ additively equal to $1\cdot \Z[\i] \oplus V \oplus W$. We define a multiplication on $B'$ by making it an $\Z[\i]$ module in the obvious way, defining multiplication by $1$ to be the identity, multiplication on $V \times W, W \times W$ to be the zero map, and letting $\varphi_k: V \times V \to w_k\Z[\i]$ be the bilinear symmetric map defined by $C_k$. We see that $B'$ is automatically commutative, and the multiplication is associative. Finally, identify $B'$ as the $\Z$-module to $\Z \oplus \Z \i \oplus \Z^{\oplus 2(n+1)} \oplus \Z^{\oplus 2m}$ and let $B$ be the submodule $\Z \oplus \left(\Z \i \oplus_{\F_2} \Z^{\oplus_{\F_2} 2(n+1)} \oplus_{\F_2} \Z^{\oplus_{\F_2} 2m} \right)$. Note $B$ is actually a subring and hence an order.

It remains to prove that $Z_{B}((X^2+1)^2)$ is non-empty if and only if $Z_{(A^*)^n}(P)$ is non-empty. Note that the subring $\Z[\i]$ is the separable subring of $B'$, and $(V+W) \cap B'$ is the nilpotent part. Let $x = a+ \sum_{k =0}^n b_k v_k + \sum_{k =1}^m c_k w_k$ with $a,b_k,c_k \in \Z[\i]$ be an element of $B$. Then if $(x^2+1)^2 = 0$, without loss of generality we have $a = \i$. We see $x^2+1$ then becomes $\sum_{k =0}^n 2 \i b_k v_k +  w$ for some $w \in W$. Letting $b = (b_0, \ldots, b_n )$, we see that $(x^2+1)^2 = -4\sum_{k=1}^m b^{\top} C_k b w_k$, and hence $B$ contains a zero of $(X^2+1)^2$ if and only if there is a $b \in \Z[\i]^{n+1}$ with $q_k(b_0,\ldots,b_n) = 0$ for every $1 \leq k \leq m$, with $b$ being $1+\i$ modulo $(2)$ on every co\"ordinate. As every $b_i$ is specifically non-zero, that is equivalent to having an $x \in \Q(\i)^{n}$ with $p_j(x_1,\ldots,x_n) = 0$ for every $1 \leq j \leq m$, with $v(x_j) = 0$ for $1\leq j \leq n$. To recap: $Z_{B}((X^2+1)^2)$ is non-empty if and only if the polynomials in $P$ have a common zero in $(A^*)^n$. This completes the proof.
\end{proof}

\bibliographystyle{alpha}
\addcontentsline{toc}{section}{References}
% \bibliography{references.bib}

% \vspace{1cm}
% For the purpose of open access, a CC BY public copyright licence is applied to any Author Accepted Manuscript (AAM) arising from this submission.

\end{document}